\documentclass[a4paper,12pt]{amsart}

\usepackage{euscript,eufrak,verbatim}
\usepackage[psamsfonts]{amssymb}
\usepackage[usenames]{color}
\usepackage[usenames]{color}
\usepackage[colorlinks,linkcolor=red,anchorcolor=blue,citecolor=blue]{hyperref}
\newtheorem{theorem}{Theorem}[section]
\newtheorem{lemma}[theorem]{Lemma}

\def\R{\mathbb{R}}

\def\Z{\mathbb{Z}}
\def\Q{\mathbb{Q}}

\def\Zp{\mathbb{Z}_{p}}
\def\Qp{\mathbb{Q}_{p}}

\def\m{\mathfrak{m}}

\begin{document}

\title[Bounded tiles in $\Q_p$ are compact open sets ]{Bounded tiles in $\Q_p$ are compact open sets}

\date{} 

\author
{Aihua Fan}

\address
{LAMFA, UMR 7352 CNRS, University of Picardie,
33 rue Saint Leu, 80039 Amiens, France}

\email{ai-hua.fan@u-picardie.fr}

\author
{Shilei Fan}

\thanks{A. H. FAN was supported by NSF of China (Grant No. 11471132) and self-determined research funds of CCNU (Grant No. CCNU14Z01002); S. L. FAN was supported by NSF of China (Grant No.s 11401236 and 11231009)}

\address
{School of Mathematics and Statistics, Central China Normal University, 430079, Wuhan, China }

\email{slfan@mail.ccnu.edu.cn}
%
%
%
%
%
%
%

\begin{abstract}
 Any bounded tile of  the field $\Q_p$ of $p$-adic numbers is a  compact open set up to a  zero Haar measure set. In this note, we give a simple and direct proof of this fact.
\end{abstract}
\maketitle

\section{Introduction}
Let $G$ be a locally compact abelian group and let $\Omega\subset G$ be a Borel set of positive and finite Haar measure.
We say that the set  $\Omega$ is a   {\em  tile} of  $G$  if there exists a set $T \subset G$ of translates
such that $\sum_{t\in T} 1_\Omega(x-t) =1$ for almost all $x\in G$, where $1_A$ denotes the indicator function of a set $A$. Such a set $T$ is called a
{\em tiling complement} of $\Omega$ and $(\Omega, T)$ is called a {\em tiling pair}.

In the  case of $G= \mathbb{R}$ considered as an additive group,  compact sets  of positive measure that tile $\R$ by translation were extensively  studied. The simplest case concerns compact sets consisting of  finite number of unit intervals all of whose endpoints are integers. This tiling problem can be reformulated in terms of   finite subsets  of $\Z$ which tile the group $\Z$.   See \cite{Tijdeman}  for references on the study of tiling problem in the group $\mathbb{Z}$.
There are also investigations on the existence of tiles   having infinitely many connected components.  For example,    self-affine tiles are studied by Bandt \cite{Ban}, Gr\"ochenig and Haas \cite{GH94}, Kenyon \cite{kenyonthesis}, Lagarias and Wang \cite{Lagarias-Wang1996}.
A structure theorem for bounded tiles in  $\R$ is obtained by Lagarias and Wang in \cite{Lagarias-WangInvent} where it is proved that all tilings of $\R$ by a bounded region (compact set with zero boundary measure) must be periodic, and that the corresponding tiling complements   are rational up to affine transformations. 

As we shall see, the situation the field of $p$-adic numbers is relatively simple.

Let $\Q_p$ be the field of $p$-adic numbers and let $\mathfrak{m}$
be the Haar measure of $\Q_p$ such that $\mathfrak{m}(\Z_p) = 1$ where $\Z_p$ is the  ring of $p$-adic integers.
A Borel set $\Omega \subset \Q_p$ is said to be {\em almost compact open} if there exists a compact open set $\Omega^{\prime}$  such that $\m(\Omega\setminus \Omega^{\prime})=\m(\Omega^{\prime}\setminus\Omega)=0$. 
In this note, we prove the following theorem.

\begin{theorem} \label{mainthm}
Assume that  $\Omega\subset \Q_p$ is  a bounded  Borel set of positive and finite  Haar measure. If $\Omega$  tiles $\Qp$ by translation, then it is an almost compact open set.
\end{theorem}

Compact open tiles in $\Q_p$ are characterized  in \cite{FFS} by the $p$-homogeneity  which is easy to check.  
Actually, it is proved in \cite{FFLS2015} that tiles in $\Q_p$ are always compact open without the boundedness assumption. But the proof in \cite{FFLS2015} was based on the theory of Bruhat-Schwartz distributions and on the theory of the Colombeau algebra of generalized functions. The proof of Theorem \ref{mainthm} given in this note is more direct and easier to understand.

\section{Proof of Theorem \ref{mainthm}}
The proof is based on Fourier analysis. We first recall some notation and some useful facts on Fourier analysis and on the $p^n$-th roots of unity. 
 
\subsection{Recall of  notation}\label{definition}
We denote by $|\cdot|_p$ the {\em absolute value} on $\Q_p$.
 A typical element of $\mathbb{Q}_p$ is expanded as a convergent series
$$
    x= \sum_{n= v}^\infty a_n p^{n} \qquad (v\in \mathbb{Z}, a_n \in \{0,1,\cdots, p-1\}, a_{v}\neq 0), 
$$
where $v$ is the valution of $x$, which satisfies $|x|_p=p^{-v}$.
We wirte $\{x\}:= \sum_{n=v}^{-1} a_n p^n$, which  is called the  {\em fractional part} of $x$. 
A non-trivial additive {\em character} of $\Q_p$ is the following function
$$
    \chi(x) = e^{2\pi i \{x\}},
$$
which produces all characters $\chi_y$ of $\mathbb{Q}_p$ $(y \in \Q_p)$, where
$\chi_y(x) =\chi(yx)$. 
Remark that  each $\chi_y(\cdot)$  is uniformly locally  constant,  i.e.
\begin{equation}\label{chi}
\chi_y(x)=\chi_y(x^{\prime}), \hbox{ if } |x-x^{\prime}|_p\leq  \frac{1}{|y|_p} . 
\end{equation}


\noindent Notation:
\\ \indent
$\Zp^\times := \Zp\setminus p\Zp=\{x\in \Qp: |x|_p=1\}$.
It is the group of units of $\Zp$.


$B(0, p^{n}): = p^{-n} \Zp$.  It is the (closed) ball centered at $0$ of radius $p^n$.

$B(x, p^{n}): = x + B(0, p^{n})$. 



$1_A:$ the indicator function of a set $A$.

$\mathbb{L}:=\left\{\{x\}, x\in \Q_p\right\}.$ It is  a complete set of representatives of the cosets of the additive subgroup $\mathbb{Z}_p$. 

%
\subsection{Fourier transformation}

Let $\mu$ be a finite Borel measure on $\Qp$. The {\em Fourier transform} of
$\mu$ is classically defined to be
$$
    \widehat{\mu} (y)
    = \int_{\Qp} \overline{\chi}_y(x) d\mu(x) \qquad (y \in \Qp).
$$
The Fourier transform $\widehat{f}$ of $f \in L^1(\Qp)$ is that of $\mu_f$ where $\mu_f$ is the measure
defined by $d\mu_f = f d\mathfrak{m}$.

The following lemma shows that the Fourier transform of the indicator function of a ball centered at $0$
is a function of the same type.
 \begin{lemma}[\cite{Fan,FFS}]\label{FourierIntegral} Let  $\gamma\in \mathbb{Z}$ be an integer. 
  We have   $$
  \forall \xi \in \Qp, \ \ \ 
  \widehat{1_{B(0, p^\gamma)}}(\xi)= p^\gamma 1_{B(0, p^{-\gamma})} (\xi).$$ 
\end{lemma}

The following lemma shows that the Fourier transform of a compactly supported integrable function is uniformly locally constant. 

 \begin{lemma}\label{localconstant} 
	Let $f\in L^1(\Q_p)$ be  a complex-value integrable function. If $f$ is supported by $B(0,p^\gamma)$, then $$  \widehat{f}(x+u)=\widehat{f}(x),\quad  \forall  x\in \Q_p  \text{ and } \forall  u\in B(0, p^{-\gamma}).$$
  \end{lemma}
\begin{proof} It suffices to observe that
	\begin{align*}
		\widehat{f}(x+u)-\widehat{f}(x)
		&=\int_{B(0,p^\gamma)}f(y)\overline{\chi(xy)}(\overline{\chi(uy)-1})dy
	\end{align*}
	and that 	 $u \in B(0, p^{-\gamma})$ implies  $|uy|_p\le 1$
	for all $y\in B(0,p^\gamma)$,  so that $\chi(uy)-1=0$, by (\ref{chi}). 
\end{proof}
\subsection{$\mathbb{Z}$-module generated by $p^n$-th roots of unity}
Let $\omega_\gamma = e^{2\pi i/p^\gamma}$, which is a primitive $p^\gamma$-th root of unity. 
\begin{lemma}[
	 \cite{Lagarias-Wang1996,Schoenberg1964}]\label{root}
Let $(a_0,a_1,\cdots, a_{p^\gamma-1})\in \Z^{p^\gamma}$.
Suppose  $$\sum_{i=0}^{p^\gamma-1 }a_i\omega_{\gamma}^i=0.$$ Then  for any integer $0\leq i\leq p^{\gamma-1}-1$,  we have $a_i=a_{i+j p^{\gamma-1}}$ for all $j=0,1,\cdots, p-1$.
\end{lemma}
 Lemma \ref{root} immediately leads to the following
 consequence.
\begin{lemma}\label{Cor-Sch}
Let $S\subset \Zp$ be a set of $p$ points and let $\mu_S = \sum_{s \in S} \delta_s$, , where $\delta_{s} $ is the dirac measure concentrated at the point $s$. Then  $\widehat{\mu_{S}}(\xi)=0$ 
if and only if  $|(s-s^{\prime})\xi|_p=p$ for all distinct $s,s^{\prime}\in S $. 
\end{lemma}
\begin{proof}
It suffices to notice that  
$
\widehat{\mu_{S}}(\xi)=\sum_{s\in S} e^{-2\pi i \{s \xi\}}.
$
\end{proof}
For  a finite set $T\subset \Zp$, let  $$\gamma_{T}:=\max_{\substack {t,t^{\prime}\in T\\t\neq t^{\prime}}}\{-\log_{p}(|t-t^{\prime}|_p)\}.$$

\begin{lemma}\label{nonzero}Let $T \subset Z_p$ be a finite set.
We have 
$\widehat{\mu_{T}}(\xi)\neq 0$   for all $ \xi \not\in B(0, p^{\gamma_T+1})$.
\end{lemma}
\begin{proof}
We give a proof by contradiction. Suppose that  $\widehat{\mu_{T}}(\xi)= 0$ for some $\xi \in \Q_p$  with $|\xi|_p>p^{\gamma_T+1} $.  Then from $\widehat{\mu_{T}}(\xi)=0$ and Lemma \ref{root}, we deduce that the cardinality of $T$ is a multiple of $p$  and  $T$ is partitioned into ${\rm Card}(T) /p$ subsets $T_0, T_1,\cdots ,T_{{\rm Card}(T)/p}$  such that  each  $T_i$  contains $p$ elements and   $\widehat{\mu_{T_i}}(\xi)= 0$. 
By Lemma \ref{Cor-Sch} applied to $S=T_0$, we get $|(t-t^{\prime})\xi|_p=p$ for distinct
   $t, t^{\prime}\in T_0$. 
However, by the definition of $\gamma_T$ and the fact $|\xi|_p>p^{\gamma_T +1}$, we have
\[
|t-t^{\prime}|_p\geq p^{-\gamma_T}>\frac{p}{|\xi|_p},
\]
which contradicts $|(t-t^{\prime})\xi|_p=p$.
\end{proof}
 
%

\subsection{Proof of Theorem \ref{mainthm}}

It is clear  that  the translation and the
dilation don't change the tiling property of a tile in $\Q_p$.  Then, without loss of generality, we assume that $\Omega\subset \Z_p$. 
Observe that for any $a\in \Q_p$, either  $\Omega+a \subset \Z_p$ or $(\Omega+a)\cap \Z_p= \emptyset$. Also observe that $\Z_p$ is a tile of $\Q_p$ with tiling complement $\mathbb{L}$ (see Section \ref{definition} for the definition of $\mathbb{L}$). These observations imply that   $\Omega$ is a  tile of $\Z_p$ if and only if   it  is a tile of  $\Q_p$.  
Now assume that  $\Omega$ is a tile  of $\Z_p$ with tiling complement $T$.  Since $\m(\Omega)>0$ and $\m(\Z_p)=1$, the tiling complement $T$ is a finite  subset of $\Z_p$.  Let $\mu_{T} := \sum_{t\in T } \delta_{t}$. By definition,  $(\Omega, T)$ is a  tiling pair in $\Z_p$ iff  the following convolution equality holds
\begin{align}\label{convo}
1_{\Omega}*\mu_{T}(x) =1_{\Zp}(x),  \quad a.e. \  x\in \Qp.
\end{align}


Taking  Fourier transform of both sides of the equality, 
by Lemma \ref{FourierIntegral}, we have 
\begin{align}\label{zero}
	 \forall\  \xi\in \Q_p\setminus\Z_p, \quad \quad \widehat{{1}_{\Omega}}(\xi)\cdot\widehat{{\mu}_{T}}(\xi)=0.
\end{align}
By Lemma \ref{nonzero},   we deduce from (\ref{zero}) that
\begin{align}\label{zero1}
\widehat{{1}_{\Omega}}(\xi)=0  \quad \hbox{for all $\xi\in  \Q_p$ with $|\xi|_p>p^{\gamma_T+1} $.}
\end{align}
Then, by using Lemma \ref{localconstant}, we get
$$ \widehat{\widehat{1_{\Omega}}}(x+u)= \widehat{\widehat{1_{\Omega}}}(x) \quad (\forall  x\in \Q_p,  \forall  u\in B(0, p^{-(\gamma_T+1)})),$$
that is to say, the continuous function $\widehat{\widehat{1_{\Omega}}}$ is constant on each ball of radius $p^{-(\gamma_T+1)}$. Therefore,
   the support of $\widehat{\widehat{1_{\Omega}}}$ is a union of disjoint balls of radius  $p^{-(\gamma_T+1)}$. 
   That is the same for $\Omega$ but up to a set of zero measure  because 
    $  \widehat{\widehat{1_{\Omega}}}(x)=1_{\Omega}(-x),  \ a.e. \  x\in \Qp. $ Then, using the fact that $\m(\Omega)<\infty $, we conclude that the support of $\Omega$ is equal  almost everywhere to a finite number of disjoint balls of radius   $p^{-(\gamma_T+1)}$.
    \qed

\setcounter{equation}{0}

\end{document}